\documentclass{article}

\usepackage{makeidx}         
\usepackage{graphicx}        
\usepackage{multicol}        
\usepackage[bottom]{footmisc}

\usepackage{amsmath}
\usepackage{amssymb}
\usepackage[ngerman,english]{babel}
\usepackage{amsthm}
\usepackage{dsfont}
\usepackage[active]{srcltx}
\usepackage{graphicx}
\usepackage{tikz}
\usepackage{pgfplots}
\newtheorem{theorem}{Theorem}[section]

\newtheorem{lemma}[theorem]{Lemma}
\newtheorem{proposition}[theorem]{Proposition}

\theoremstyle{definition}

\renewcommand\P{\mathbb{P}}

\newcommand\R{\mathbb{R}}
\newcommand\1{\mathds{1}}

\newcommand{\N}{\mathbb{N}}
\newcommand{\Z}{\mathbb{Z}}
\renewcommand\d{\, \mathrm{d}}

\renewcommand{\phi}{\varphi}

\newcommand{\F}{\mathcal{F}}
\newcommand{\abs}[1]{\left|#1\right|}
\renewcommand{\F}{\mathcal{F}}
\newcommand{\nF}{\mathcal{F}}

\newcommand{\menge}[1]{\left\lbrace #1\right\rbrace }

\newcommand{\p}{\partial}
\newcommand{\eps}{\varepsilon}
\newcommand{\dd}{\hspace{1pt}{\rm d}\hspace{0.5pt}}
\newcommand{\ee}{{\rm e}\hspace{1pt}}

\newcommand{\EE}{{\rm E}}

\date{}

\begin{document}

\title{Transport in a stochastic Goupillaud medium}
\author{Florian Baumgartner\footnote{Institute of Mathematics, University of Innsbruck, Technikerstra\ss e 13, 6020 Innsbruck, Austria, {florian.baumgartner@uibk.ac.at}}\\
Michael Oberguggenberger\footnote{Unit of Engineering Mathematics, University of Innsbruck, Technikerstra\ss e 13, 6020 Innsbruck, Austria, {michael.oberguggenberger@uibk.ac.at}}\\
Martin Schwarz\footnote{Unit of Engineering Mathematics, University of Innsbruck, Technikerstra\ss e 13, 6020 Innsbruck, Austria, {martin.schwarz@uibk.ac.at}}}
\maketitle

{\bf Abstract}{ This paper is part of a project that aims at modelling wave propagation in random media by means of Fourier integral operators. A partial aspect is addressed here, namely explicit models of stochastic, highly irregular transport speeds in one-dimensional transport, which will form the basis for more complex models. Starting from the concept of a Goupillaud medium (a layered medium in which the layer thickness is proportional to the propagation speed), a class of stochastic assumptions and limiting procedures leads to characteristic curves that are L\'evy processes. Solutions corresponding to discretely layered media are shown to converge to limits as the time step goes to zero (almost surely pointwise almost everywhere). This translates into limits in the Fourier integral operator representations.}

\section{Introduction}
\label{sec:intro}

This contribution is part of a long-term project that aims at modelling wave propagation in random media by means of Fourier integral operators. The intended scope includes, for example, the equilibrium equations in linear elasticity theory
\[
\rho\frac{\p^2u_i}{\p t^2} - \sum_{j,k,l}\frac{\p}{\p x_j}c_{ijkl}\frac{\p u_l}{\p x_k} = f_i,\quad i = 1,2,3
\]
or, more generally, hyperbolic systems of the form
\begin{equation}\label{eq:hyp}
\frac{\p u}{\p t} = \sum_{j=1}^m A_{j}(t,x)\frac{\p u}{\p x_j} + B(t,x)u + f(t,x),
\end{equation}
to be solved for the unknown functions $u(t,x) = (u_1(t,x),\cdots, u_n(t,x))$. Here $t$ denotes time, $x$ is an $m$-dimensional space variable, and $A_j$, $B$ are $(n\times n)$-matrices.

Our specific interest is in the situation where the coefficient matrices are random functions of the space variable $x$, i.e., random fields. Such a situation typically arises in seismology (propagation of acoustic waves) or in material science (damage detection). There are many ways of setting up models for random fields (see e. g. \cite{GhanemSpanos:1991, Matthies:2008}), but typically random fields describing randomly perturbed media have continuous, but not differentiable paths. As coefficients in hyperbolic equations such as (\ref{eq:hyp}), this degree of regularity is too low and does not allow one to apply the classical solution theory for hyperbolic equations. In addition, the solution depends nonlinearly on the coefficients, so it is generally impossible to directly calculate the stochastic properties of the solution from knowledge of the distribution of the coefficients.

The main thrust of the project will be to write the solution to equations like (\ref{eq:hyp}) as a sum of Fourier integral operators
\begin{equation}\label{eq:FIO}
u(t,x) = \frac1{(2\pi)^m}\iint\ee^{i\phi(t,x,y,\eta)}a(t,x,y,\eta)u_0(y)\, \dd y\dd\eta
\end{equation}
applied to the initial data $u_0$. This is possible in the case of deterministic, smooth coefficients (up to a smooth error). The ultimate goal of the project will be to set up the stochastic model of the medium through the phase function $\phi$ and the amplitude $a$ of the Fourier integral operator, rather than through a direct stochastic model of the coefficients, as described in \cite{OberSchwarz:2014}.

A second thrust is in understanding wave propagation in strongly irregular stochastic media with a sufficiently simple structure and tractable properties, in order to get insight into what stochastic processes are suitable to be entered as phase functions and amplitudes. This brings us to the topic of this paper, namely, wave propagation in a \emph{Goupillaud medium} (the name goes back to \cite{Goupillaud:1961}). In this contribution, we will work out the case of one-dimensional transport under assumptions that will lead to characteristic curves given by an increasing L\'evy process with possibly infinitely many jumps on each subinterval.

One-dimensional transport is described by the equation
\begin{equation} \label{eq:trans}
\begin{array}{rcl}
\dfrac{\p}{\p t} u(t,x) + c(x) \dfrac{\p}{\p x} u(t,x) & = & 0\\[10pt]
u(0,x)  & = & u_0(x)
\end{array}
\end{equation}
The material properties of the medium are encoded in the transport speed $c(x)$. The Goupillaud assumption is that $c(x)$ is a piecewise constant function so that the travel time in each layer is the same. That is, the thickness of layer number $k$ is proportional to the propagation speed $c_k$ in that layer.

Further, the propagation speeds $c_k$ will be given by independent, identically distributed random variables. At this stage, various choices of the type of random variables as well as scalings are possible. For the wave equation, such scalings leading to fairly regular limiting processes have been introduced in \cite{BurridgeGoupillaud:1988} and studied in \cite{FouqueETAL:2007, NairWhite:1991}. Our procedure of dyadic refinements on the time axis will lead to infinitely divisible, positive random variables. It turns out that they can be constructed as increments of a strictly increasing L\'evy process, a so-called subordinator with positive drift. As the time step goes to zero, the characteristic curve of (\ref{eq:trans}) passing through the origin is a path of a L\'evy process.
We will show that the characteristic curves of the discrete Goupillaud medium converge (almost surely at almost every $(x,t)$) to limiting curves (actually translates of the obtained L\'evy process), and that the corresponding solutions and their Fourier integral operator representations converge as well.

The limiting function $u(x,t)$ is constant along the limiting characteristic curves, as in the case of classical transport. However, the limiting characteristics may possibly have infinitely many jumps on each interval. Due to this high degree of singularity, we cannot give a meaning to the limiting function $u(x,t)$ as a solution to (\ref{eq:trans}) -- it is just a limit of piecewise classical solutions. This situation is quite common in the theory of singular stochastic partial differential equations, see e.g. \cite{Hairer:2014}.

A few remarks about the regularity of the coefficient $c(x)$ in (\ref{eq:trans}) is in order. If the coefficient is Lipschitz continuous, classical solutions can be readily constructed. If the coefficient is a piecewise constant, positive function, piecewise classical solutions are obtained easily. In case of lower regularity of the coefficient, various approaches have been proposed in the literature. We mention the work of DiPerna and Lions \cite{DiPernaLions:1989}, Bouchut and James \cite{BouchutJames:1998}, and Ambrosio et al. \cite{Ambrosio:2004, Ambrosio:2008} in the deterministic, $x$-dependent case; for a recent survey, see Haller and H\"ormann \cite{HallerHoermann:2008}. In the stochastic case, recent work of Flandoli \cite{Flandoli:2011} shows how solutions can be constructed adding noise in the transport term. Finally, another different line of development is constituted by extending the reservoir of generalized functions, either in the direction of white noise analysis or in the direction of Colombeau theory. A representative article pursuing and comparing both approaches is by Pilipovi\'{c} and Sele{\v{s}}i \cite{PilipovicSelesi:2010b, PilipovicSelesi:2010a}.

The plan of the paper is as follows: In the first part, the stochastic Goupillaud medium is set up and analyzed. In the second part, the limiting behavior as the time step goes to zero is established. The paper ends with some conclusions and open questions.

\section{Setting up the Goupillaud medium}
\label{sec:Goupillaud}

If the initial data $u_0$ are differentiable and the propagation speed $c$ is Lipschitz continuous, classical solutions to the transport equation (\ref{eq:trans}) can be readily obtained by the method of characteristics. The characteristic curves are the integral curves of the vector field $\p/\p t + c(x) \p/\p x$ passing through the point $x$ at time $t$, that is, the solutions to the ordinary differential equation
\[
   \frac{\d }{\d \tau}{\gamma}(\tau;x,t) = c(\gamma(\tau;x,t)),\quad \gamma(t;x,t)=x.
\]
Then the solution to (\ref{eq:trans}) is given by
\[
  u(x,t)= u_0(\gamma(0;x,t)).
\]
Under the mentioned assumptions, the function $u$ is continuously differentiable, and the solution is unique in this class.

If the speed parameter $c$ is constant the characteristic curves are simply given by $\gamma(\tau;x,t)=x+c(\tau-t)$. If the parameter is piecewise constant one can compute the characteristic curves as polygons. Assuming continuity across interfaces, the solution $u$ is given as a continuous, piecewise differentiable function, which solves (\ref{eq:trans}) in the weak sense.

\subsection{Dyadic deterministic structure}
\label{subsec:dyadic}

We begin by setting up the discrete, deterministic Goupillaud medium.
Take an equidistant sequence $(t_j)_{j\in\N}$ of points of time with $t_0=0$ and time step $\Delta t\equiv t_j-t_{j-1}$ for all $j=1,2,\ldots$. Furthermore, take a strictly increasing sequence $(x_k)_{k\in\Z}$ with $x_0=0$ and $x_k\to\pm \infty$ as $k\to \pm \infty$ and let $\Delta x_k = x_{k}-x_{k-1}$. The coefficient $c(x)$ is defined as
\begin{align} \label{speedparameter}
c(x) = \sum_{k=-\infty}^{\infty} \frac{\Delta x_k}{\Delta t} \1_{[x_{k-1},x_{k})}(x).
\end{align}
In other words, the time for passing a layer $\Delta x_k$ is constant, namely $\Delta t$. For an illustration see Figure~1.
\begin{figure}[htb] 	\label{fig:Gitter_grob}
	\centering
	\includegraphics[width = 0.49\textwidth]{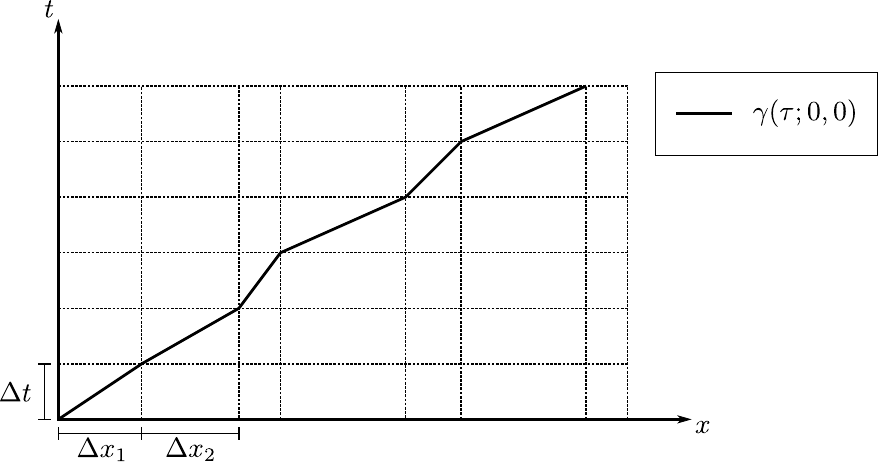}
	\includegraphics[width = 0.49\textwidth]{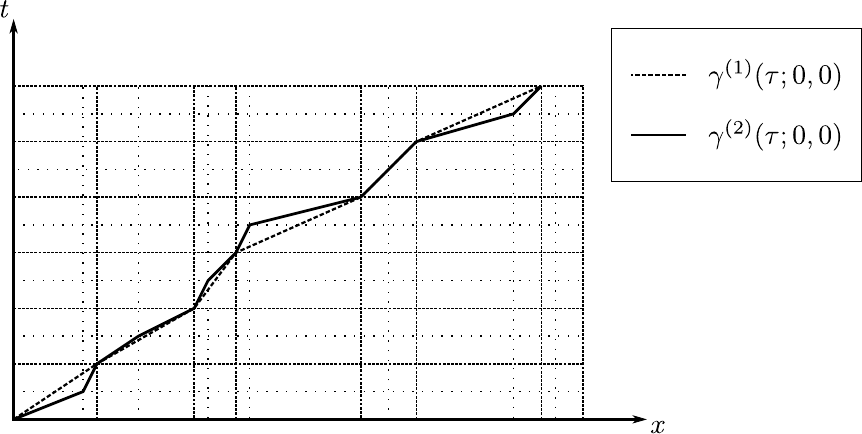}
	\caption{Illustration of $\tau\rightarrow \gamma(\tau;0,0)$ in the Goupillaud medium (left). Refining the grid on the right preserves existing grid points $\gamma^{(1)}(\cdot;0,0)$ (dashed) and refinement $\gamma^{(2)}(\cdot;0,0)$.}
\end{figure}

Call $c_k$ the value of $c(x)$ in the $k$-th layer, that is, $x_{k-1}\leq x < x_{k}$. Then the Goupillaud relation
\begin{align}\label{relationship:c:deltat}
 {\Delta x_k} = c_k\Delta t
\end{align}
holds for all $k$, with constant $\Delta t$. The structure of the Goupillaud medium makes computing the values of the characteristic curves $\gamma(\tau;x,t)$ in the grid points very simple. In fact,
\begin{align}\label{eq:grid0}
\gamma\left(t_j;x_k, t_l\right)  = x_{j+k-l}
\end{align}
for all integers $j,k,l$. Since every point $(x,t)$ is just a convex combination of the neighboring grid points, the values $\gamma(\tau;x,t)$ can be easily obtained anywhere.

We now set up a dyadic refinement of the initial grid. Define
\[
 \Delta t^{(N)}= 2^{-N},\quad t^{(N)}_j=\frac{j}{2^N}
\]
and let $x^{(N)}_k\in \R$, $k\in\Z$, be a strictly increasing sequence of spatial points (or equivalently, propagation speeds $c^{(N)}_k > 0$ satisfying
$\Delta x^{(N)}_k = c^{(N)}_k \Delta t^{(N)}$). We require that each resulting grid is a dyadic refinement of the previous one, that is
\begin{equation}\label{dyadicref}
 \big(t^{(N+1)}_{2j},\, x_{2k}^{(N+1)}\big)=\big(t^{(N)}_{j},\, x_{k}^{(N)}\big),
\end{equation}
as illustrated in Figure~1. This condition implies
$$ \Delta x^{(N)}_k= \Delta x^{(N+1)}_{2k-1}+ \Delta x^{(N+1)}_{2k}.$$
Inductively, one obtains
\begin{align} \label{eq:consistency}
\quad \Delta x_{k}^{(N)}= \sum_{i=1}^{2^M} \Delta x_{(k-1)2^M+i}^{(N+M)}
\end{align}
for all $N,M\in\N,k\in\Z$.
The value of the characteristic curve $\gamma^{(N)}$ in the grid points is readily obtained according to (\ref{eq:grid0}). For any integer $N,j,k,l$ we have
\begin{align}\label{eq:grid}
\gamma^{(N)}\left(t_j^{(N)};x_k^{(N)}, t^{(N)}_l\right)  = x^{(N)}_{j+k-l}.
\end{align}
For any  $N\in \N$ and $\tau\in [t_{k-1}^{(N)}, t_{k}^{(N)})$, the characteristic curve through the origin $\xi^{(N)}(\tau)=\gamma^{(N)}(\tau;0,0)$ can be represented as
\begin{equation}\label{eq:lambda}
\xi^{(N)}(\tau)= \alpha^{(N)} (\tau) \xi^{(N)}(t_{k-1}^{(N)})+ \big(1-\alpha^{(N)} (\tau) \big)\xi^{(N)}(t_{k}^{(N)}),
\end{equation}
where
$$\alpha^{(N)}(\tau) =\frac{t_{k}^{(N)}-\tau}{t_{k}^{(N)}-t_{k-1}^{(N)}}=\big(t_{k}^{(N)}-\tau\big)2^N\in [0,1].$$
and $\xi^{(N)}(t^{(N)}_k)=x^{(N)}_k$ by \eqref{eq:grid}.

In other words, $\xi^{(N)}$ is an increasing polygon through $(t_k^{(N)},x_k^{(N)})$, $k\in\Z$.
For $(x,t)\in\R^2$ one obtains the characteristic curve through $(x,t)$ by
\begin{equation}\label{eq:charN}
\gamma^{(N)}(\tau;x,t)= \xi^{(N)}\big(\tau+(\xi^{(N)})^{-1}(x)-t\big),\quad \tau\in\R,
\end{equation}
i.e., by shifting $\xi^{(N)}$ in time direction such that it passes through $(x,t)$.

\subsection{The stochastic model}
\label{subsec:stochmod}

In this subsection, we formulate the stochastic assumptions underlying our model of a randomly layered medium, in which $\Delta x_k^{(N)}$ is random. In the sequel, we will denote random elements by capital letters and realizations by the corresponding small ones. Let $(\Omega,\F,\P)$ be a probability space which is rich enough. Our decisive assumption is that for each $N\in\N$, the increments are positive i.i.d.\ random variables $\Delta X_k^{(N)}$, $k\in\Z$,

\smallskip
Together with our previous consistency assumption \eqref{eq:consistency}, this implies that $\Delta X_{k}^{(N)}$ is infinitely divisible for every $k\in\Z$ and $N\in\N$ (using e.g. \cite[Thm. 15.12]{Kallenberg:2002}). Let $\mu =\P\circ (\Delta X^{(0)}_1)^{-1}$ be the distribution of $\Delta X^{(0)}_1$. Then $\Delta X_{k}^{(N)}\sim \mu^{*1/2^N}$, $k\in\N$, the $2^N$-th unique root of $\mu$, cf.\ \cite[p.~34]{Sato:1999}. Again by \cite[Thm. 15.12]{Kallenberg:2002}, there exists a L\'evy process $X=(X(t))_{t\in\R}$ on a probability space, w.l.o.g.\ say $(\Omega,\nF,\P)$, with $\P\circ X(1)^{-1}=\mu$, that is, $\Delta X^{(0)}_1$ has the same distribution as $X(1)$.
These conditions are met, e.g., by Poisson processes or Gamma processes with positive drift.

\smallskip
Having derived the L\'evy process, we may use it as a starting point for defining the stochastic Goupillaud medium. We let $t_k^{(N)} = k/2^N$ as in subsection~\ref{subsec:dyadic} and define
$$
\Delta X_{k}^{(N)}= X\big(t_k^{(N)}\big)-X\big(t_{k-1}^{(N)}\big)
$$
and
$$
   X_{k}^{(N)} = \sum_{i=1}^k \Delta X_{i}^{(N)} = X\big(t_k^{(N)}\big)
$$
for $k>0$ and similarly for $k\leq 0$. The consistency condition \eqref{eq:consistency} is clearly satisfied.
Let furthermore $X^{(N)}(\omega,\cdot)$ be the piecewise affine interpolation of $X(\omega,\cdot)$ through the grid points $(t_k^{(N)},X_k^{(N)})$ as in \eqref{eq:lambda}. This construction is carried out pathwise for fixed $\omega\in\Omega$.

%


\section{Limits as the time step goes to zero}
\label{sec:convergence}

The main result of this section is that the characteristic curves of the discrete Goupillaud medium converge to limiting curves (almost surely almost everywhere). This will imply that the solutions to the transport equation converge to a limit as well (in a sense to be made precise). The crucial observation is that the paths of a L\'evy process are c\`adl\`ag almost surely, i.e., they are continuous from the right and have left-hand limits.

\subsection{A convergence result for c\`adl\`ag functions}
\label{subsec:cadlag}

The first convergence result holds generally for c\`adl\`ag functions. Thus let $t\to\xi(t)$ be an increasing c\`adl\`ag function with $\xi(t)\rightarrow \pm \infty$ for $t\rightarrow \pm \infty$ and let $(x,t)\in\R^2$. Set
$$
\xi^*(x)= \inf\menge{t\in\R: \xi(t)\geq x}
$$
which is Borel measurable, and
\[
\gamma(\tau;x,t)= \xi(\tau+\xi^*(x)-t),\quad \tau\in\R.
\]
Further, let $\xi^{(N)}$ be a piecewise linear interpolation of $\xi$ that coincides with $\xi$ at the grid points $t_k^{(N)} = k/2^N$, $k \geq 0$, and define $\gamma^{(N)}(\tau;x,t)$ by formula (\ref{eq:charN}).

\begin{lemma}\label{lem:cadlag}
	Let $(x,t)\in \R^2$, $\xi$, $\xi^{(N)}$, $\gamma$, $\gamma^{(N)}$ as described above.
 	If the function $\tau\to\gamma(\tau;x,t)$ does not have a jump in $\tau_0$, then
	\[ \lim_{N\to\infty}\gamma^{(N)}(\tau_0;x,t)= \gamma(\tau_0;x,t),\]
	i.e., $\gamma^{(N)}(\,\cdot\,; x,t)$ converges pointwise to $\gamma(\,\cdot\,;x,t)$ at the points of continuity of $\gamma(\,\cdot\,; x,t)$.
\end{lemma}

\begin{proof}
		 Fix $\varepsilon>0$ and $R>|\tau_0|+|\xi^*(x)|+|t|$ and define
		\begin{align*}
		s&:=\tau_0+\xi^*(x)-t\\
		s^{(N)}&:=\tau_0+(\xi^{(N)})^{-1}(x)-t
		\end{align*}
As $\xi$ is c\`adl\`ag there exist finitely many $(t_1,\ldots,t_\ell)\in[-R,R]$ such that
\begin{equation}\label{eq:Billingsley}
\forall r_1,r_2\in [t_i,t_{i+1}):\abs{\xi(r_1)-\xi(r_2)}<\frac{\varepsilon}{3},
\end{equation}
see e.g. \cite[Lemma~1, p.~110]{Billingsley:1999}. Since $\gamma$ is continuous in $\tau_0$ we can assume without loss of generality that $s\neq t_i$ for all $i$.
		
Since $\xi$ and $\xi^{(N)}$ coincide at the grid points and both are increasing, $(\xi^{(N)})^{-1}(x)$ and $\xi^*(x)$ belong to the same interval of length $1/2^N$, for every $N$. It follows that $|s^{(N)}-s|<1/2^N$ as well. We can choose $N$ large enough, such that both $s$ and $s^{(N)}$ belong to $(t_i + 1/2^N,t_{i+1}-1/2^N)$ for some
$i\in \{1,\ldots,\ell-1\}$. From (\ref{eq:Billingsley}) we get that
\[
  \abs{\gamma^{(N)}(\tau_0;x,t)-\gamma(\tau_0;x,t)} = \abs{\xi^{(N)}(s^{(N)})-\xi(s)}
         \leq  \abs{\xi^{(N)}(s^{(N)})-\xi(s^{(N)})} + \frac{\eps}3.
\]
Now choose $k$ such that
\[
   t_{i} < t_{k-1}^{(N)} \leq s^{(N)} \leq t_{k}^{(N)} < t_{i+1}.
\]
Using (\ref{eq:lambda}) and the fact that $\xi$ and $\xi^{(N)}$ coincide at all grid points $t_{j}^{(N)}$ we can write
\begin{eqnarray*}
  \xi^{(N)}(s^{(N)})-\xi(s^{(N)}) &=&
    \alpha^{(N)}(s^{(N)})\xi(t_{k-1}^{(N)}) + \big(1-\alpha^{(N)}(s^{(N)})\big)\xi(t_{k}^{(N)})-\xi(s^{(N)})\\
    &=& \big(1-\alpha^{(N)}(s^{(N)})\big)\big(\xi(t_{k}^{(N)}) - \xi(t_{k-1}^{(N)})\big)
        + \big(\xi(t_{k-1}^{(N)})-\xi(s^{(N)}\big).
\end{eqnarray*}
Recalling $\alpha^{(N)}\in[0,1]$ for all $N$ and invoking again (\ref{eq:Billingsley}) we conclude that
\[
  \abs{\gamma^{(N)}(\tau_0;x,t)-\gamma(\tau_0;x,t)} < \eps,
\]
which implies the desired convergence. \hfill \qed
\end{proof}

Denote by $T = (T_i)_{i\in\N}$ the (countable) set of jump points of the c\`adl\`ag function $\xi$. At fixed $(x,t)$, convergence may fail at those values $\tau$ for which $\tau+\xi^*(x)-t\in T$. This exceptional set is countable, but may be different for every $(x,t)$. Next, we fix $\tau_0$ and determine the set of all $(x,t)$ for which convergence fails. We are going to show that its two-dimensional Lebesgue measure $\lambda$ is zero.

\begin{lemma}\label{lem:zeromeasure}
Let $\tau_0\geq 0$; $\xi$, $\gamma$ as in Lemma\;\ref{lem:cadlag}. The set $M = \{(x,t): \gamma(\tau;x,t)\ {\rm jumps}\ {\rm at}\ \tau_0\}$ has Lebesgue measure zero.
\end{lemma}

\begin{proof} Letting
\[
  M = \{(x,t): \tau_0 + \xi^*(x)-t \in T\} = \bigcup_{i\in\N}M_i
\]
where $M_i=\{(x,t): \tau_0 + \xi^*(x)-t = T_i\}$, it suffices to check that each $M_i$ has Lebesgue measure zero. But each $M_i$ is jointly measurable, and for each $x$ the set $M_i(x)=\{t: \tau_0 + \xi^*(x)-t = T_i\}$ is a singleton. Hence
\[ \lambda(M_i) = \int_\R \lambda(M_i(x))\d x =0 \]
by Fubini's theorem. \hfill \qed
\end{proof}

%
\subsection{Convergence of characteristic curves}
\label{subsec:convchar}

We now apply Lemma\;\ref{lem:cadlag} to a path $t\to X(\omega,t)=\xi(t)$ of the L\'evy process $X$ constructed in subsection\;\ref{subsec:stochmod}. Since L\'evy processes are c\`adl\`ag almost surely, there is $\Omega_0\in \F$ with $\P(\Omega_0)=1$ such that $X(\omega,\cdot)$ is c\`adl\`ag for all $\omega\in\Omega_0$.
With the notation of subsections\;\ref{subsec:dyadic} and \ref{subsec:stochmod}, let
\[
\Gamma^{(N)}(\omega; \tau; x,t)=X^{(N)}\big(\omega;\tau+(X^{(N)}(\omega))^{-1}(x)-t\big)
\]
and
\[
\Gamma(\omega; \tau; x,t)=X\big(\omega;\tau+X(\omega)^*(x)-t\big).
\]
%
%
\begin{proposition}\label{prop:conv}
	Let $\omega\in\Omega_0$.
	\begin{itemize}
		\item[(1)] \ For a.e.\ $(x,t)\in \R^2$ it holds that
		$$\Gamma^{(N)}(\omega;0;t,x)\rightarrow \Gamma(\omega;0;t,x)\quad{\rm as}\quad N\to\infty.$$
		
		\item[(2)] \  Let $u_0$ be bounded and continuous, $1\leq p < \infty$ and $K\subseteq \R^2$ compact. Then
		\begin{align}
			\int_K \left| u_0\big(\Gamma^{(N)}(\omega;0;t,x)\big)-u_0\big(\Gamma(\omega;0;t,x)\big)\right|^p \d \lambda(x,t)\rightarrow 0
		\end{align}
   as $N\to\infty$.
		\item[(3)] \  $\Gamma^{(N)}(\cdot;0;\cdot,\cdot)\to\Gamma(\cdot;0;\cdot,\cdot)$ as $N\to \infty$ with convergence $\P\otimes \lambda$-a.e.
		\item[(4)] \  Let $u_0$ be bounded and continuous,  $1\leq p < \infty$ and $K\subseteq \R^2$ compact. Then
		\begin{align}
			\int_{\Omega\times K} \left| u_0\big(\Gamma^{(N)}(\omega;0;t,x)\big)-u_0\big(\Gamma(\omega;0;t,x)\big)\right|^p \d (\P\otimes\lambda)(\omega,x,t)\rightarrow 0
		\end{align}
		as $N\to \infty$.
		\end{itemize}
\end{proposition}	
\begin{proof} First we notice that $\Gamma(\omega; \tau; x,t)=X\big(\omega;\tau+X(\omega)^*(x)-t\big)$ is jointly measurable in all variables. This follows by measurability of the mapping $\R\times \mathcal{D}\rightarrow \R,\;(t,\xi)\to \xi(t)$, cf.\ \cite[p.~132]{Sato:1999}, where $\mathcal{D}$ is the space of c\`adl\`ag functions endowed with the $\sigma$-algebra generated by the coordinate mappings $\xi\to \xi(t)$, $t\in\R$.
	
(1) follows from Lemma\;\ref{lem:zeromeasure}, and (2) follows from Lebesgue's convergence theorem.

(3) We first convince ourselves that the exceptional set of those $(\omega,x,t)$ at which $\Gamma(\omega; \tau; x,t)$ jumps at $\tau = 0$ is jointly measurable, that is
\begin{eqnarray*}		
\mathcal{M} & = & \menge{(\omega,x,t)\in \Omega_0\times \R^2: \Gamma(\omega;0; x,t) - \Gamma(\omega; 0-; x,t) \neq 0}\\
   & = & \menge{(\omega,x,t)\in \Omega_0\times \R^2: X\big(\omega;X(\omega)^*(x)-t\big)\neq \lim_{\tau\to 0-}X\big(\omega;X(\omega)^*(x)-t + \tau\big)}
\end{eqnarray*}
is $\F\otimes\mathcal{B}(\R^2)$-measurable. It follows from the joint measurability of $X\big(\omega;\tau+X(\omega)^*(x)-t\big)$ that the function $d: \Omega_0\times \R^2\to \R$,
\[
	d(\omega,x,t) = X\big(\omega;X(\omega)^*(x)-t\big) - \lim_{\tau\to 0-}X\big(\omega;X(\omega)^*(x)-t + \tau\big)
\]
is measurable and therefore $\mathcal{M}=d^{-1}(\R\setminus \menge{0})\in \F\otimes \mathcal{B}(\R^2)$.

The fact that $\mathcal{M}$ has measure zero, hence (3), is a consequence of Fubini's theorem. Indeed, for $\omega\in\Omega_0$, let
$M(\omega)=\mathcal{M}\cap (\{\omega\}\times \R^2)$. By Lemma\;\ref{lem:zeromeasure} applied with $\tau_0 = 0$ it follows that $M(\omega)$ has Lebesgue measure zero. Also, $\P(\Omega_0) = 1$. Thus
$\P\otimes \lambda(\mathcal{M})=\int_{\Omega}\lambda(M(\omega))\d \P(\omega) =0$.

(4) follows immediately from (2) and (3). \hfill\qed
\end{proof}

\subsection{Convergence of approximate solutions}
\label{subsec:convsolu}

We return to the transport equation in the discrete stochastic Goupillaud medium
\begin{equation} \label{eq:transN}
\begin{array}{rcl}
\dfrac{\p}{\p t} U^{(N)}(\omega;t,x) + C^{(N)}(\omega;x) \dfrac{\p}{\p x} U^{(N)}(\omega;t,x) & = & 0\\[10pt]
U^{(N)}(0,x)  & = & u_0(x)
\end{array}
\end{equation}
with\begin{align} \label{speedparameterN}
C^{(N)}(\omega;x) = \sum_{k=-\infty}^{\infty} \frac{\Delta X_k^{(N)}(\omega)}{\Delta t^{(N)}} \1_{[X_{k-1}^{(N)}(\omega),X_{k}^{(N)}(\omega))}(x)
\end{align}
where $\Delta X_k^{(N)}$ is derived from the L\'evy process $X$ as in subsection\;\ref{subsec:stochmod}. To be precise about the solution concept, assume that $u_0$ belongs to the Sobolev space $W^{1,1}_{\sf loc}(\R)$. Note that this implies that $u_0$ is a continuous function. At fixed $\omega$, the transport coefficient $C^{(N)}(\omega;\cdot)$ is a piecewise constant, locally bounded function, and the characteristic curves $\Gamma^{(N)}(\omega; \tau; x,t)$ are piecewise linear, continuous functions. We put
\[
    U^{(N)}(\omega;t,x) = u_0\big(\Gamma^{(N)}(\omega; 0; x,t)\big).
\]
It is straightforward to check that $U^{(N)}(\omega;\cdot,\cdot)$ belongs to $W^{1,1}_{\sf loc}(\R^2)$ and is continuous. Taking weak derivatives in the sense of $W^{1,1}_{\sf loc}(\R^2)$ and performing the multiplication with the $L^\infty_{\sf loc}$-function $C^{(N)}(\omega;\cdot)$ in $L^1_{\sf loc}(\R^2)$ shows that $U^{(N)}(\omega;\cdot,\cdot)$ satisfies the equation (\ref{eq:transN}) in the sense of the latter space. Further, the initial data are taken as continuous functions. In this sense, $U^{(N)}(\omega;\cdot,\cdot)$ is a pathwise solution to (\ref{eq:transN}). Define
\[
   U(\omega;t,x) = u_0\big(\Gamma(\omega; 0; x,t)\big).
\]
With the results of subsection \ref{subsec:convchar} we are now in the position to formulate convergence of the approximate solutions
$U^{(N)}$ to $U$.
\begin{proposition} Let $u_0 \in W^{1,1}_{\sf loc}(\R)$. Then
\begin{itemize}
\item[(1)]\ $\lim_{N\to\infty} U^{(N)}(\omega;t,x) = U(\omega;t,x)$ pointwise $\P\otimes \lambda$-a.e. and
\[
   \lim_{N\to\infty} {\EE}^\P\|U^{(N)} - U\|_{L^p(K)} = 0
\]
whenever $K$ is a compact subset of $\R^2$ and $1\leq p < \infty$.

\item[(2)] \ If the Fourier transform of $u_0$ belongs to $L^1(\R)$, then $U$ has the Fourier integral operator representation
\[
   U(\omega;t,x) = \frac1{2\pi}\iint \ee^{i(\Gamma(\omega;t,x)-y)\eta}u_0(y)\, \dd y\dd\eta.
\]
\end{itemize}
\end{proposition}
\begin{proof}
(1) is evident from Proposition\;\ref{prop:conv}. Concerning (2), observe that
\[
   U^{(N)}(\omega;t,x) =
      u_0\big(\Gamma^{(N)}(\omega; 0; x,t)\big) = \frac1{2\pi}\iint \ee^{i(\Gamma^{(N)}(\omega;t,x)-y)\eta}u_0(y)\, \dd y\dd\eta
\]
as shown by taking Fourier transforms, where the double integral converges as an iterated integral. Proposition\;\ref{prop:conv} allows us to take the limit as $N\to\infty$ inside the integral, whence the assertion. \hfill \qed
\end{proof}
Note that a priori there is no meaning for $u$ to be a solution of the transport equation\;\ref{eq:trans} other than being a limit of approximate solutions.

For the sake of illustration, we show two realizations of the limiting solutions. The initial value $u_0$  is taken as a triangular function, the realizations of $U$ are shown at times $t = 1,2,3$. We use two different L\'evy processes as drivers $X$ (cf. subsection~\ref{subsec:stochmod}). In the first picture in Figure~2, $X$ is taken as a Gamma process, in the second picture, $X$ is a Poisson process, both with positive drift. The solutions have constant parts, which are created if the L\'evy process jumps at this point.
\begin{figure}[htb]\label{fig:U_realization}
	\centering
\includegraphics[width = 0.8\textwidth]{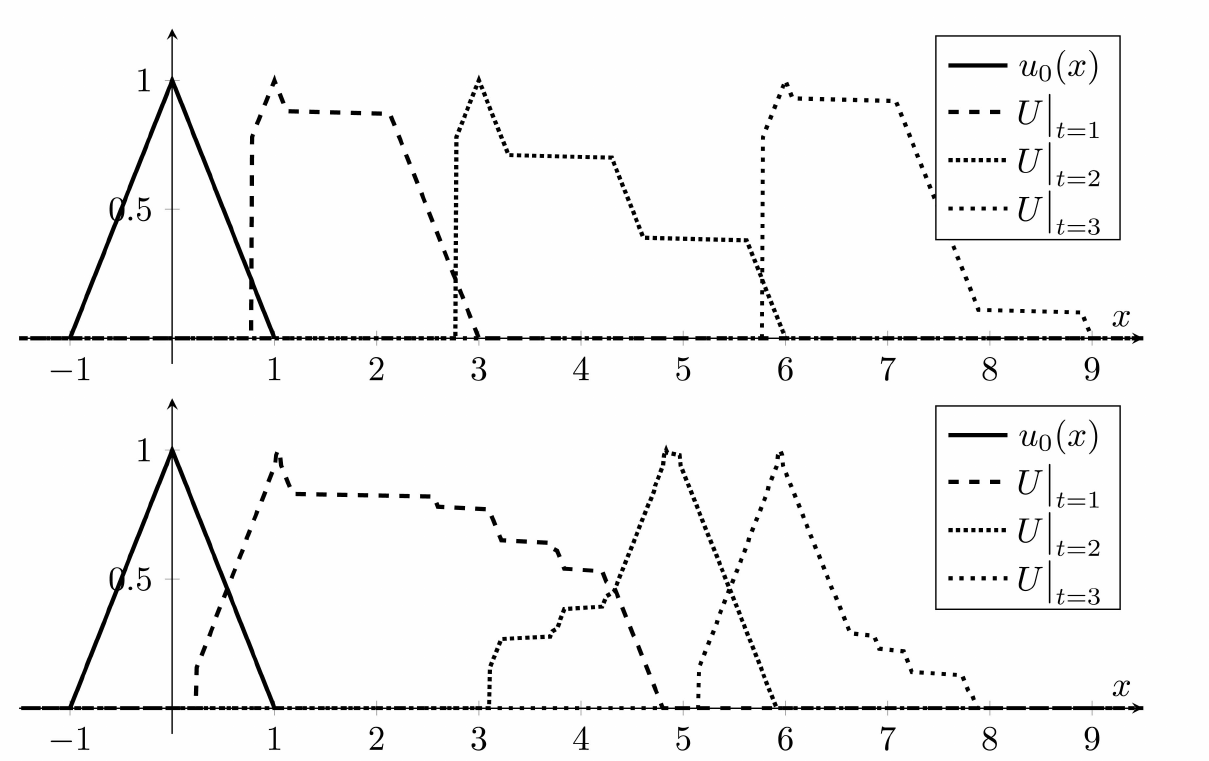}
		\caption{ A trajectory of the  solution $U(x,t)$ at several time points. Above: The generating L\'evy process $X$ is Gamma distributed with scale parameter $k=1$, shape parameter $\theta=1$ and drift $d=1$. Below: The generating L\'evy process $X$ is Poisson distributed with jump size 1, intensity $c=1$ and drift $d=1$.}
\end{figure}

\section{Conclusion}

A Goupillaud medium is a piecewise constant layered medium such that the thickness of each layer is proportional to the corresponding propagation speed. We have developed a set-up for a specific stochastic Goupillaud medium in which the propagation speeds (or equivalently the thickness of the layers) are given by infinitely divisible random variables. Using a dyadic refinement, these random variables could be constructed as increments of a strictly increasing L\'evy process. We have shown that the one-dimensional transport equation can be solved in such a medium, and that the characteristic curves converge to shifted trajectories of the underlying L\'evy process as the time step goes to zero. If the initial data are sufficiently regular, the corresponding solutions converge pathwise and in the $p$-th mean to a limiting function, which in addition can be computed by means of a Fourier integral operator.

At this stage, several questions remain open. The first issue is the probability distribution of the limiting characteristic curves $\Gamma(\omega; \tau; x,t)=X\big(\omega;\tau+X(\omega)^*(x)-t\big)$, and subsequently of the limiting solution $U(\omega;t,x) = u_0\big(\Gamma(\omega; 0; x,t)\big)$. The second question is how one can give a meaning to the limiting propagation speed $c(x)$ as a (generalized) function of $x$. Given a positive answer to this question, one may finally ask if there is solution concept that would allow one to interpret $U(\omega;t,x)$ as a solution in some sense. All these issues are the subject of ongoing research.

\subsection*{Acknowledgement}
The second author acknowledges support through the research project P-27570-N26 ``Stochastic generalized Fourier integral operators''
of FWF (The Austrian Science Fund).
The third author acknowledges support through the Bridge Project No. 846038 ``Fourier Integral Operators in Stochastic Structural Analysis'' of FFG (The Austrian Research Promotion Agency).

\bibliographystyle{plain}

\end{document}